\documentclass[12pt]{amsart}
\usepackage{amssymb,latexsym,exscale,enumerate}
\usepackage[T1]{fontenc} 

\newcommand{\G}{\mathcal G}

\newcommand{\m}{\mathfrak m}

\newcommand{\R}{\mathcal R}

\newcommand{\calS}{\mathcal S}
\newcommand{\calT}{\mathcal T}
\newcommand{\V}{\mathcal V}

\newcommand{\bN}{\mathbb N}
\newcommand{\bQ}{\mathbb Q}
\newcommand{\bR}{\mathbb R}

\def\Ind#1#2{#1\setbox0=\hbox{$#1x$}\kern\wd0\hbox to 0pt{\hss$#1\mid$\hss}
\lower.9\ht0\hbox to 0pt{\hss$#1\smile$\hss}\kern\wd0}
\def\ind{\mathop{\mathpalette\Ind{}}}
\def\Notind#1#2{#1\setbox0=\hbox{$#1x$}\kern\wd0\hbox to 0pt{\mathchardef
\nn=12854\hss$#1\nn$\kern1.4\wd0\hss}\hbox to
0pt{\hss$#1\mid$\hss}\lower.9\ht0 \hbox to
0pt{\hss$#1\smile$\hss}\kern\wd0}
\def\nind{\mathop{\mathpalette\Notind{}}}

\def\thind{\mathop{\mathpalette\Ind{}}^{\text{\th}} }
\def\nthind{\mathop{\mathpalette\Notind{}}^{\text{\th}} }

\newcommand{\acl}{\mathrm {acl}}

\newcommand{\Aut}{\mathrm {Aut}}
\newcommand{\dcl}{\mathrm {dcl}}

\newcommand{\eq}{\mathrm {eq}}
\newcommand{\Eq}[1]{\mathrel{#1 ^{\eq}}}
\newcommand{\id}{\mathrm {id}}
\newcommand{\kInt}{\mathrm {kInt}}

\newcommand{\rc}{\mathrm {rc}}
\newcommand{\res}{\mathrm {res}}
\newcommand{\rk}{\mathrm {rk}}
\newcommand{\RV}{\mathrm {RV}}
\newcommand{\Th}{\mathrm {Th}}
\newcommand{\tp}{\mathrm {tp}}

\newcommand{\ball}[3]{\mathrel{B^{\mathrm #1}_{#2}(#3)}}
\newcommand{\Ball}[2]{\mathrel{B_{#1}(#2)}}

\theoremstyle{plain}
\newtheorem{theorem}{Theorem}[section]
\newtheorem{lemma}[theorem]{Lemma}
\newtheorem{corollary}[theorem]{Corollary}
\newtheorem{proposition}[theorem]{Proposition}

\theoremstyle{definition}

\newtheorem{definition}[theorem]{Definition}
\newtheorem{fact}[theorem]{Fact}
\newtheorem{example}[theorem]{Example}

\theoremstyle{remark}

\newtheorem*{unnumberedclaim}{Claim}
\numberwithin{claim}{theorem}

\newenvironment{proof of claim}
{\begin{trivlist}  \item \textit{Proof of Claim.}~} {\hfill $\Box$ (Claim)
\end{trivlist}}

\begin{document}

\title{Residue Field Domination in Real Closed Valued Fields}

\author[C. Ealy]{Clifton Ealy}
\address{Department of Mathematics, Western Illinois University, 1 University Circle, Macomb, Illinois 61455, USA}
\email{CF-Ealy@wiu.edu}
\thanks{}

\author[D. Haskell]{Deirdre Haskell}
\address{Department of Mathematics and Statistics\\
	  McMaster University \\
	  1280 Main St W. \\
         Hamilton ON L8S 4K1, Canada} 
\email{haskell@math.mcmaster.ca}
\thanks{The second author was partially supported by Discovery Grants from the Natural Sciences and Engineering Research 
Council of Canada, and wishes to 
thank MSRI, the CIMS of New York University and the Mathematics Department of the University of California Berkeley for hospitality
during some of the work.}

\author[J. Ma\v{r}\'{i}kov\'{a}]{Jana Ma\v{r}\'{i}kov\'{a}}
\address{Department of Mathematics, Western Illinois University, 1 University Circle, Macomb, Illinois 61455, USA}
\email{J-Marikova@wiu.edu}
\thanks{The third author was supported by a grant from the Simons Foundation (318364, J.M.).}


\date{\today}

\begin{abstract}
We define a notion of residue field domination for valued fields which generalizes stable domination in algebraically
closed valued fields. We prove that a real closed valued field is dominated by the sorts internal to the residue field,
over the value group, both in the pure field and in the geometric sorts. These results characterize forking and \th-forking
in real closed valued fields (and also algebraically closed valued fields). We lay some groundwork for extending these
results to a power-bounded $T$-convex theory. 
\end{abstract}
\maketitle


\section{Introduction}

The notion of domination of a theory by its stable part was developed by Haskell, Hrushovski and Macpherson in \cite{HHM2} and 
illustrated in the case of an algebraically closed valued field. It follows from the elimination of imaginaries in \cite{HHM1} that the stable
part of an algebraically closed valued field consists essentially of vector spaces over the algebraically closed residue field. In 
\cite[Chapter 12]{HHM2} it is shown that an algebraically closed valued field is dominated over its value group by
its residue field. Our goal here is to prove an appropriate analogue of these results for a certain class of ordered valued fields.
In an ordered valued field (that is, an ordered field with a convex valuation) both the value group and the residue field are ordered, 
so there is no stable part. Instead we consider a notion of domination by the residue field over the value group. We  prove the 
following theorems (the terminology is defined in the subsequent pages).

\begin{enumerate}[1)]
	\item  Over a maximal base, an ordered valued field is dominated by the value group and 
	residue field (Corollary~\ref{dom-by-valuegp-resfield}). 
	\item  Over a maximal base and its value group, a real closed valued field is dominated 
	by the sorts internal to the residue field (Corollary~\ref{dom-by-resfield-over-valuegp}).
	\item  Over a maximal base, forking and \th-forking are determined by the value group and residue field 
	(Theorem~\ref{forking-over-maximal}). 
	\item  An expansion of a maximal base in the geometric sorts is dominated by the value group and 
	residue field, and over its value group is dominated by the sorts internal to the residue field (Theorem~\ref{sorted-domination}).
\end{enumerate}
We hope to be able to extend these results to the more general situation of substructures of power-bounded $T$-convex structures
in some future work.  For this reason we work in the $T$-convex setting (see \cite{vdDL}) whenever possible.

Let $R$ be any o-minimal field (that is, an o-minimal expansion of a real closed field) in a language $\mathcal{L}_{o}$.  Expand 
$\mathcal{L}_{o}$ to $\mathcal{L}$ by adding a unary predicate $V$ for a convex subring (which we will also refer to as $V$), and 
consider $R$ as an $\mathcal{L}$-structure.  When the reduct $R|_{\mathcal{L}_o}$ is a pure real 
closed field, $\Th(R)$ is the theory of real closed valued fields, RCVF. This theory has elimination of quantifiers in an appropriate language 
\cite{CD}, and elimination of imaginaries  in an appropriate sorted language \cite{M}. In the more general setting where 
$R|_{\mathcal{L}_o}$ is an o-minimal field with theory $T$, one gets similar results if the structure $R$ is {\em $T$-convex}; that is, if 
$f(V)\subseteq V$ for all continuous functions $f\colon R\to  R$ that are $\emptyset$-definable in $\mathcal{L}_o$. In this case, $\Th(R)$ 
is universally axiomatisable (in $\mathcal{L}$ expanded by a constant for an element of 
$R^{>V}$), has quantifier elimination relative to quantifier elimination in $R|_{\mathcal{L}_o}$, and  there
is a natural way to make the residue field into a model of $T$ (see \cite{vdDL}).  It is not yet known what is a minimal sorted language
in which $\Th(R)$ has elimination of imaginaries (see \cite{HHM3}). Stronger statements about the definable sets can be made when
$R$ is {\em power-bounded}; that is, if for each definable $f\colon R\to R$ there is $\lambda$ in the field of exponents of $R$ such 
that $|f(x)|<x^\lambda$ for all sufficiently large $x$. Furthermore, if an o-minimal field is power-bounded, then every model of its 
theory is power-bounded (see for example \cite[p.~23]{vdD2} for more 
details on power-boundedness).

For most of the paper, we work in a monster model $\mathcal{R}^{\eq}$, where $\mathcal{R}$ is an $\mathcal{L}$-structure such that its 
$\mathcal{L}_o$-reduct is a power-bounded o-minimal field, and $\V$ is a $T$-convex subring (although for most results we will restrict to 
the special case of a pure RCVF). In what follows we shall refer to such structures $\R$ as power-bounded $T$-convex structures. The 
quotient $k \colon = \V/\m$, where $\m$ is the maximal ideal of $\V$, is the residue 
field of $\R$, and we write  $\res \colon \V\to k$ for the natural map. The ordered abelian group
$\R^{\times}/(\V \setminus \m)$ is the value group $\Gamma$, and we write $v\colon \R^{\times} \to \Gamma$ for the natural map. 
The valuation is also well-defined on $\RV := \R^\times/(1+\m)$, which is viewed as a 
multiplicative group. Associated to $\RV$ is an exact sequence of abelian groups (note that while $k^{\times}$ and $\RV$ are 
multiplicative groups, $\Gamma$ is additive):
$$1 \to k^{\times} \to \RV \to \Gamma \to 0,$$ where 
the map $k^{\times} \to \RV$ is inclusion (note that $x+\m =x(1+\m )$ for $x\in \R$ with $v(x)=0$), and $\RV \to \Gamma$ is the map 
$x(1+\m ) \mapsto v(x)$, where $x \in \R^{\times}$. Setwise, $x(1+\m ) \in \RV$ is just the open ball around $x$ of radius $v(x)$. 
For any $\gamma\in \Gamma$, we write $\RV_\gamma$ for the fiber above $\gamma$, and observe that  
$\RV_0$ is definably isomorphic to $k^{\times}$.
 
For any substructure $K$ of $\R$, we write  $V_K$ for its valuation ring, $\Gamma_K$ 
for its value group and $k_K$ for its residue field. Algebraic closure and definable closure are always taken in the
model-theoretic sense and in the sorted structure $\R^{\eq}$, unless stated otherwise. Note that since $\R$, $k$, and $\Gamma$ 
are ordered, algebraic closure in these sorts is equal to definable closure. For sets $A$, $B$ and a subfield $C$ of $\mathcal{R}$, 
we write $C[A,B]$ for the ring generated by $A$ and $B$ over $C$,  $C(A,B)$ for the field generated by 
$A$ and $B$ over $C$, and $C\langle A,B\rangle$ for the o-minimal structure generated by $A$ and $B$ over $C$. We shall
also write $\Gamma(A)=\dcl(A)\cap \Gamma$ and $k(A)=\dcl(A)\cap k$. In general, for any collection of sorts 
$\calS$,  we  write $\calS(A) = \dcl(A)\cap\calS$.

Given a valued field $K$, we shall write 
\begin{align*}
	\ball{op}{\gamma}{a} = \{x\in K : v(x-a) > \gamma\}, \\
	\ball{cl}{\gamma}{a} =  \{x\in K : v(x-a) \ge \gamma\}
\end{align*}
for the open (respectively closed) ball of radius $\gamma$ around $a$, where $a\in K$, $\gamma\in \Gamma_K$.
If we do not want to specify whether the ball is open or closed, we write $\Ball{\gamma}{a}$.

We use the following results about definable sets in $T$-convex theories.

\begin{fact}\label{definablesets} \cite[Theorem~12.10]{T}, \cite[Proposition~7.6]{vdD2} Let $R$ be $T$-convex with 
$T$ power-bounded and let  $S\subseteq R$ be a definable set. Then $S$ is a finite boolean combination of points, 
intervals, and open and closed balls defined over $R$.
\end{fact}

\begin{fact}\label{residue-field-embedded} \cite[Theorem A]{vdD2}
Let $R$ be $T$-convex, and let $S\subseteq R^n$ be a definable set. Then $\res(S)\subseteq k^n$
is definable in $k$, considered as a model of $T$.
\end{fact}

\begin{fact}\label{value-group-embedded} \cite[Theorem B]{vdD2}
Let $R$ be $T$-convex with $T$ power-bounded, and let $S\subseteq {(R^\times)}^n$ be a definable set. Then
$v(S)\subseteq {\Gamma_R}^n$ is definable in $\Gamma$, considered as an ordered vector 
space over the field of exponents.
\end{fact}

In a pure valued field, the dimension inequality relates the transcendence degree of a field extension to the degrees 
of the extensions of the value group and residue field. In a power-bounded $T$-convex theory, the proof of the corresponding result, 
called the Wilkie Inequality, is much more subtle. 

\begin{fact}\label{Wilkieinequality}\cite[Theorem C]{vdD2} 
Let $T$ be power-bounded with field of exponents $F$, and let $R\preceq S$ be $T$-convex structures. Then
\[
	\rk(S/R) \ge \rk(k_{S}/k_{R}) + \dim_F(\Gamma_{S}/\Gamma_{R}) .
\]
\end{fact}

Now suppose  that $R\preceq R\langle a\rangle$ are power-bounded $T$-convex structures, and $a \in \mathcal{R}$ is a singleton.   If the
extension $R\langle a \rangle$ is not immediate, then by Fact~\ref{Wilkieinequality}, exactly  one of the value group and residue
field increases. There is a precise description of the increased  value group or residue field.

\begin{fact}\label{non-immediate-extensions}
Let $R\preceq R \langle a\rangle$ be models of a power-bounded $T$-convex theory.
\begin{enumerate}[i)]
	\item If $\res(a)\notin k_R$ then $k_{R\langle a\rangle} = k_R\langle \res(a)\rangle$ \cite[Lemma 5.1]{vdDL}.
	\item Assume $T$ is power bounded with field of exponents $F$. If $v(a)\notin \Gamma_R$ then 
	$\Gamma_{R\langle a\rangle} = \Gamma_R \oplus F\cdot v(a)$ \cite[Lemma 5.4]{vdD2}.
	\end{enumerate}
\end{fact}

Furthermore, in a power-bounded $T$-convex theory, the residue field and value group are orthogonal, in the sense of the 
following statement.
\begin{fact}\label{orthogonality}\cite[Proposition~5.8]{vdD2}
 Let $R$ be a model of a power-bounded $T$-convex theory. Any definable function from $k_R$ to $\Gamma_R$
 or from $\Gamma_R$ to $k_R$ has finite image.
 \end{fact}

Our general goal is to study the extent to which types in the valued field are controlled by their restriction to different 
sorts in the structure. 
We call this property {\em domination}. It generalizes the notion of stable domination in \cite{HHM2} and is related to the notion
of compact domination \cite{HPP}.  The reader will note the definition of domination is reminiscent of the uniqueness of nonforking 
extensions in a stable theory.  Thus, in an extremely imprecise sense, the existence of domination is an instance of stable-like behavior 
in the structure.
 
Let $\calS$ and $\calT$ be collections of sorts which are stably embedded in all models of some theory.
Assume that $\calS$ and $\calT$ are {\em rosy} and hence have some notion of independence. 

\begin{definition}\label{domination-definition} Let $C\subseteq A$ be sets of parameters.  We say that $\tp(A/C)$ is {\em dominated by 
$\mathcal S$} if whenever  
$B\supset C$ with $\calS(B)$ independent from  $\calS(A)$ over $\calS(C)$, one has that $\tp(A/C\calS(B))$ implies $\tp(A/CB)$.  
Further, we say that $\tp(A/C)$ is {\em dominated by $\calS$ over $\calT$} if $\tp(A/C\calT(A))$ is dominated by $\calS$.
\end{definition}

We will also express the property of domination using automorphisms, as follows.

\begin{fact}\label{domination-automorphism-definition}
 Let $A$, $B$, $C$, $\calS$, and $\calT$ be as above, with $\calS(B)$ independent from  $\calS(A)$ over $\calS(C)$. Then the following 
 are equivalent:
  \begin{enumerate}[i)]
    \item $\tp(A/C \calT (A)\calS (B))$ implies $\tp(A/C \calT (A) B)$;
    \item given an automorphism $\sigma$ of $\R$ fixing $C\calT(A)\calS(B)$, there is an automorphism agreeing with $\sigma$ 
    on $A$ and fixing $C\calT(A)B$.
  \end{enumerate}
\end{fact}

In Corollary~\ref{dom-by-valuegp-resfield}, we use this terminology with $\calS = \Gamma\cup k$. In 
Corollary~\ref{dom-by-resfield-over-valuegp}, $\calS$ consists of (some of) the sorts internal to the residue field and $\calT$ is the value 
group. In both cases, the notion of independence is given by \th-forking (we recall the definition in Section~3, or see \cite{adler}). 
We will only need to use  \th-independence applied to parameters from $k$, $\Gamma$, or $\RV_\gamma$, all of which are o-minimal, 
which allows us, for most of the paper, to use the more concrete description below \cite[Section~6]{EO}.   

\begin{definition}\label{o-minimal-basis-and-rank}
Given any (possibly infinite) tuple of parameters $(a_i)$ from a definable set whose induced structure is o-minimal, and a set of parameters 
$C$, a subtuple $(a_{i_j})$ is said to be an {\em o-minimal basis} of $(a_i)$ over $C$ if $(a_i) \in \acl(C(a_{i_j}))$ and the $a_{i_j}$ are 
algebraically independent over $C$.  When the length of $(a_{i_j})$ is finite, this length is called the {\em rank} 
of the tuple $(a_i)$ over $C$.  Any definable set has rank equal to the maximal rank of a tuple from the definable set, and a type has 
rank equal to the rank of any tuple realizing the type.
\end{definition}

\begin{fact}\label{thorn-forking-fact}
 The following are equivalent when the parameters $A$ come from a set with induced o-minimal structure:
  \begin{enumerate}[i)]
    \item $A\thind_C B$;
    \item for every finite tuple $a$ from $A$, the rank of $\tp(a/CB)$ equals the rank of $\tp(a/C)$.
  \end{enumerate}
\end{fact}

\begin{example}\label{ominimal-thorn}
In a real closed field, algebraic independence in the sense of model theory is the same as field theoretic algebraic independence. 
Thus $A\thind_C B$ if and only if any tuple from $\acl(CA)$ which is algebraically dependent over $\acl(CB)$ is also algebraically 
dependent over $\acl(C)$. It is also equivalent to say that  $\acl(CA)$ is linearly disjoint from $\acl(CB)$ over $\acl(C)$; the 
algebraic closure of the base is necessary here.  Clearly, in a divisible ordered 
abelian group, $A\thind_C B$ if and only if $\acl(CA) \cap \acl(CB) = \acl(C)$.
\end{example}

\noindent Combining Definition~\ref{domination-definition} with Fact~\ref{domination-automorphism-definition}, 
we see that $\tp(A/C)$ is dominated by $k$  over $\Gamma$ if, whenever one has $B$ with $k(B) \thind_{k(C)} k(A)$, and an 
automorphism $\sigma:A\to A'$ fixing $C\Gamma(A)k(B)$, then there is an automorphism $\tau:A\to A'$ fixing $C\Gamma (A) B$.  
Note that $k(C)=k(C\Gamma(A))= k(C\Gamma(A'))$ by Fact~\ref{orthogonality}.

\smallskip
Recall that a definable set $E$ is {\em internal} to
a definable set $D$ if there is a  finite set $A$ such that $E\subseteq \dcl(D\cup A)$.

\begin{definition}\label{residue-field-internal}
Given a set $B$, and $S\subseteq \Gamma$, 
we define $$\kInt_{S}^B = \acl(k(B) \{ \RV_{\gamma} (B)\}_{\gamma \in S}).$$
\end{definition}

\noindent As $k(B)$ and $\RV_{\gamma} (B)$, where $\gamma \in S$, are all internal to the residue field, so is $\kInt_{S}^B$.
In the ACVF case,  $\acl(C\kInt_{\Gamma_L}^M)$ can be shown to be precisely the part of $M^{\eq}$ which is internal to the residue 
field and contained in sets definable over $C$ and $\Gamma(L)$ \cite[12.9]{HHM2}.  This may also be true in the RCVF case, but we 
have not investigated it. Since the residue field is stably embedded, and the function witnessing the internality of $RV_\gamma$ uses 
parameters from within $RV_\gamma$, $\kInt_{S}^B$ is also stably embedded.  

In the case where $B$ is a model, any element of $\RV_{\gamma}(B)$ is definable over $k_B$ from any other element. This is clearly 
true for $\gamma =0$, and is obtained for arbitrary $\gamma$ using the $b(1+\m )$-definable bijection $\RV_{0} (B)\to \RV_{\gamma} (B)$ 
given by $x(1+\m )\mapsto b x (1+\m )$, where $v(b)=\gamma$ and $v(x)=0$. Thus, if for each $\gamma \in S$, we let $d_\gamma$ 
be an element of  $\RV_{\gamma} (B)$, then
\[
	\kInt_{S}^B = \acl(k_B \{d_{\gamma}\}_{\gamma \in S}) .
\]

The following proposition corresponds to \cite[12.9, 12.10]{HHM2} in the context of algebraically closed valued fields.
The proof given there uses Morley rank, which is not available to us in the ordered field context.

\begin{proposition}\label{real-prop-12.10}
Let $\mathcal{R}$ be a power-bounded $T$-convex structure with field of exponents $F$, and let 
$L$ and $M$ be substructures of $\mathcal{R}$. Let $C$ be a common substructure of $L$ and $M$, 
and suppose that $\Gamma_L \subseteq \Gamma_M$.  The following are equivalent.
 \begin{enumerate}[i)]
	 \item For some (equivalently any) choice of elements $(a_i)$ and $(b_j)$ from $L$ and $(e_i)$ from $M$ such that $(v(a_i))$ 
	 is an $F$-basis of $\Gamma(L)$ over $\Gamma(C)$, $(\res(b_j))$  is an o-minimal basis of $k(L)$ over $k(C)$, and for each
	 $i$, $v(e_i)=v(a_i)$, and for all finite subtuples $a_1, \dots, a_r$,  $b_1, \dots, b_s$, the sequence
\[
	\res(a_1/e_1), \dots,\ res(a_r/e_r), \res(b_1), \dots, \res(b_s)
\]
is algebraically independent over $k_M$ (in the model-theoretic sense);
	\item $\kInt_{\Gamma_L}^L$ is algebraically independent from $\kInt_{\Gamma_L}^M$ over $C\Gamma_L$;
	  \item $\kInt_{\Gamma_L}^L \thind_{C\Gamma_L} \kInt_{\Gamma_L}^M$.
\end{enumerate}
\end{proposition}

\begin{proof}
Let $(a_i)$, $(b_j )$, and $(e_i)$ be sequences as in the hypothesis of the proposition.
For each $i$, we let $d_i=\ball{op}{v(a_i)}{a_i}$. 
\begin{unnumberedclaim}
The sequence  $d_1,\ldots,d_r,\res(b_1),\ldots,\res(b_s)$ is algebraically independent over $C \Gamma_L$.
\end{unnumberedclaim}
\begin{proof of claim}
We first show algebraic independence over $C$. 
Assume to the contrary that $d_{i_0}\in\acl(C(d_i)_{i\ne i_0}(b_j))$. Then $v(a_{i_0})\in \Gamma(C(d_i)_{i\ne i_0}(b_j))$. However, 
it follows from Fact~\ref{non-immediate-extensions} ii) that
\[
	\Gamma(C(d_i)_{i\ne i_0}(b_j)) = \Gamma(C) \oplus F((v(a_i)_{i\ne i_0}) .
\]
This contradicts the independence of $(a_i )$. 
Similarly, it follows from Fact~\ref{non-immediate-extensions}~i)
that 
\[
	k(C(d_i)(b_j)_{j\ne j_0}) = k(C)\langle \res(b_j)_{j\ne j_0} \rangle .
\]
 If some $\res(b_{j_0})\in \acl(C,(d_i)(b_j)_{j\ne j_0})$ then 
$\res(b_{j_0}) \in k(C)\langle \res(b_j)_{j\ne j_0} \rangle$, which would contradict the assumption of independence on the $\res(b_j )$.

We now show that $d_1,\ldots,d_r,\res(b_1),\ldots,\res(b_s)$ is algebraically independent over $C\Gamma_L$.
If not, then there is some $k$ with $1\le k \le r$, an element $x\in \{d_1,\ldots,d_r,\res(b_1),\ldots,\res(b_s)\}$
 and a non-constant function $f$ definable with parameters in 
 $\{d_1,\ldots,d_r,\res(b_1),\ldots,\res(b_s)\}\setminus \{x\} \cup \{v(a_1), \dots, v(a_{k-1})\}$ with $f(v(a_k))=x$.
This would contradict Fact~\ref{orthogonality}.

\end{proof of claim}

Now we prove the equivalence of the three statements. The equivalence of ii) and iii) is Fact~\ref{thorn-forking-fact}. Note that the o-minimal 
rank of $k(C b_1,\ldots,b_s )$ over $k(C)$  
equals $s$, so by the Wilkie inequality, $$\rk(\Gamma(Cb_1,\ldots,b_s)/\Gamma(C)) = 0.$$ 
By the above claim, $\kInt_{\Gamma_L}^L$ is algebraically independent from $\kInt_{\Gamma_L}^M$ over 
$C\Gamma_L$  if and only if some (equivalently every) sequence $d_1, \ldots, d_r,$ $\res(b_1),$ $ \ldots, \res(b_s)$  has 
rank $r+s$ over $\kInt_{\Gamma_L}^M$, which, by stable embeddedness of 
$\kInt_{\Gamma_L}^M$, happens if and only if the sequence $d_1,\ldots,d_r,$ $\res(b_1)$, $\ldots,\res(b_s)$  
has rank $r+s$ over $M$. As each $d_i$ is interdefinable over $M$ with $\res(a_i/e_i)$,  the sequence 
$d_1,\ldots,d_r,$ $\res(b_1)$, $\ldots,\res(b_s)$  has rank $r+s$ over $M$ if and only if
$\res(a_1/e_1),\ldots,$ $\res(a_r/e_r),\res(b_1),$ $\dots,\res(b_s)$ 
has rank $r+s$ over $M$,  which is equivalent to i).
\end{proof}

\section{Isomorphism theorems}

We begin with some remarks about separated bases over a maximal subfield.

\begin{definition} 
Let $C\subseteq \R$ be a valued field. We say that a finite sequence $m_1,\ldots,m_n \in \R$, is {\em separated over $C$} if 
$$v(\sum_{i=1}^n c_i m_i) = \min\{v(c_i) + v(m_i):1\le 1\le n \}$$ 
for all $c_1,\ldots,c_n\in C$. 
\end{definition}

\begin{definition}
A valued field $C$ is {\em maximal} if it has no proper immediate extension. It is {\em spherically complete} if every 
non-empty chain of balls defined over $C$ has an element in $C$. 
\end{definition}

The fact that a field is maximal if and only if it is spherically complete dates to Kaplansky and Ostrovski (see \cite[Theorem 8.28]{K}). 
The Hahn field $\bR((\bQ))$ is an example of a maximal real closed field. It follows from results in \cite{T} that any power-bounded 
$T$-convex theory has a maximal model; details of the proof are not relevant for this article.
The definition of spherical completeness can be extended to a finite vector space over the valued field $C$ in the following way.

\begin{lemma}\label{spherically-complete-group}
Let $C$ be a spherically complete valued field, and assume that the tuple $m=(m_1,\ldots,m_n)$ is separated over $C$.
Then  the valued group $$C^n \cdot m =\{ \sum_{i=1}^n c_i m_i: c_i\in C\}$$ is spherically complete, in the sense that every
non-empty chain of balls with centers in $C^n \cdot m$ and radii in $v(C^n \cdot m)$ has an element in $C^n \cdot m$.
\end{lemma}

\begin{proof}
The proof is a straight-forward calculation.
\end{proof}

It is also well-known that a finite-dimensional vector space over a maximal field has a separated basis  
(see \cite[Lemma 2.6]{H} for a proof). 
Indeed, the following stronger statement is true, as shown in the proof of \cite[Proposition 12.11]{HHM2}.

\begin{fact}\label{factseparated}
Let  $L$  and $M$ be valued field extensions of the maximal
valued field $C$ such that $\Gamma(L)\cap \Gamma(M) =\Gamma(C)$ and $k_L$ and $k_M$ are linearly 
disjoint over $k_C$. Then, for any finite-dimensional subspace $U$ of $M$ over $C$, there is a separated basis
$m = (m_1,\ldots,m_n)$ for $U$ over $C$ which is also separated over $L$. 
Furthermore,
\begin{align*}
	 v(L^n  \cdot m) \cap \Gamma(M) = \bigcup_{1\leq i \leq n} \big( v(m_i) + \Gamma(C) \big). 
\end{align*}
\end{fact}

For the rest of this section we assume that $\R$ is a pure real closed valued field. We now prove 
our first domination statement, in the formulation involving automorphisms.

\begin{theorem}\label{real-prop-12.11}
Let $C$, $L$, $M$ be ordered valued fields with $C$ maximal and $C\subseteq L \cap M$. Assume that 
$\Gamma_L\cap \Gamma_M = \Gamma_C$ and $k_L$ is linearly disjoint from $k_M$ over $k_C$. Suppose there
are automorphisms $\sigma$ and $\tau$ of $\R$ fixing $C\Gamma_L k_L$ with $\sigma(L)=L'$, $\tau(M)=M'$. Then there
is an automorphism of $\R$ which maps $N=C(L,M)$ to $N'=C(L',M')$ and restricts to $\sigma$ on $L$ and $\tau$ 
on $M$. Furthermore, $\Gamma_N$ is the group generated by $\Gamma_L$ and $\Gamma_M$, and $k_N$ is
the field generated by $k_L$ and $k_M$.
\end{theorem}

\begin{proof}
We first assume that $\tau=\id$.  We extend $\sigma|_L$ to $C[L,M]$ by setting
$\sigma (\sum_{i=1}^n \ell_i m_i ) = \sum_{i=1}^n \sigma(\ell_i) m_i$, and then to the field of fractions $N$ in the only possible way. 
We need to show that the extended function, also called $\sigma$, preserves the valuation and the ordering. It will then
extend to an automorphism of $\R$ by quantifier elimination.
The fact that $\sigma$ is a valued field isomorphism is proved in \cite[Proposition~12.11]{HHM2}, and 
it is also shown there that $\Gamma_N$ is the group generated by $\Gamma_L$ and $\Gamma_M$, and that $k_N$
is the field generated by $k_L$ and $k_M$.  Here we need to show that $\sigma$ also preserves the ordering. 
We thank Tom Scanlon for pointing out the argument below, which is much simpler than our original proof.

Suppose for contradiction that there is $a\in C[L,M]$ with $a> 0$, $\sigma(a)<0$. First suppose that $a=\ell m$ for some 
$\ell \in L$ and $m\in M$. Dividing by $m$, we see that this contradicts the assumption that $\sigma$ preserves the ordering
on $L$.  Now let $a=\sum_{i=1}^n \ell_i m_i$. As $\Gamma_N$ is the group generated by $\Gamma_L$ and $\Gamma_M$,
$v(a) = v(\ell m)$ for some $\ell \in L$, $m\in M$. Write $a=\ell m a_1$, where $a_1\in C[L,M]$ and $v(a_1)=0$. As 
$\sigma(a) = \sigma(\ell m) \sigma(a_1)$ and $\sigma (\ell m)$ does not change sign (as already noted), it must be that
$a_1$ and $\sigma(a_1)$ have opposite sign. But then $\res(a_1)$ and $\res(\sigma(a_1))=\sigma(\res(a_1))$ have 
opposite sign, which contradicts the fact that $\sigma$ is the identity on $k_N$, as it is generated by $k_L$ and $k_M$.

Now suppose $\tau$ is any automorphism fixing $Ck_L\Gamma_L$ and such that $\tau:M\to M'$.  Apply $\tau^{-1}$ to $N'=C(L',M')$ to 
get a valued field $\widetilde N = C(\widetilde L, M)$ for $\widetilde L = \tau^{-1}(L')$. As $\Gamma_L = \Gamma_{L'} = \Gamma_{\widetilde L}$
and $k_L = k_{L'} =k_{\widetilde L}$, the hypotheses of the theorem apply to $\widetilde L$ and $M$, so we may apply the case of the
theorem that we have already proven to deduce that there is an isomorphism from $N$ to $\widetilde N$ which restricts to  $\tau^{-1}$
on $L$ and the identity on $M$. Composing this isomorphism with $\tau$ gives us the desired map $N\to N'$.
\end{proof}

\noindent We restate the theorem in the language of domination.

\begin{corollary}\label{dom-by-valuegp-resfield}
Let $C\subseteq M$ be substructures of $\R$ with $C$ maximal, $k_M$ a regular extension of $k_C$ and 
$\Gamma_M / \Gamma_C$  torsion free.Then $\tp(M/C)$ is dominated by the value group and residue field.
\end{corollary}

\begin{proof}
Let $L\supseteq C$ be another substructure and assume that $k_{L} \Gamma_{L} \thind_{k_C \Gamma_{C}} k_{M} \Gamma_{M}$. We need 
to show that $\tp(M/C\Gamma_L k_L) \vdash \tp(M/L)$. That is, as in Fact~\ref{domination-automorphism-definition}, for any automorphism 
$\tau$ of $\R$ fixing $C\Gamma_L k_L$ there is an automorphism agreeing with $\tau$ on $M$ and fixing $L$. This is the conclusion of 
Theorem~\ref{real-prop-12.11}. The hypotheses of the theorem are satisfied because the \th-independence of the residue fields implies 
their algebraic independence over $\acl(k_C)$ (and thus over $k_C$), 
and hence their linear disjointness over $k_C$ by the regularity assumption (see \cite[Theorem VIII 4.12]{Lang}). Similarly, the \th-independence 
of the value groups  implies, as observed in Example~\ref{ominimal-thorn}, that $\Gamma_L\cap\Gamma_M\subseteq\dcl(\Gamma_C)$.  
But the definable closure of $\Gamma_C$ is just the divisible hull of $\Gamma_C$, and by our assumption on torsion 
$\dcl(\Gamma_C)\cap \Gamma_M=\Gamma_C$.  Thus $\Gamma_L\cap\Gamma_M=\Gamma_C$.
\end{proof}

Corollary~\ref{dom-by-valuegp-resfield} has the following consequence in the special case that $L$ is an immediate extension of the 
real closure $C^\rc$ of $C$. 
Despite the fact that $C$ is not necessarily dense in its real closure, the order type of $M$ with respect to $L$ is determined by its 
order type with respect to $C$. This applies in particular when $L$ is a maximal extension of $C^\rc$.

\begin{corollary}\label{up-to-real-closure}
Let $C\subseteq M$ be substructures of $\R$ with $C$ maximal, $k_M$ a regular extension of $k_C$ and $\Gamma_M / \Gamma_C$
torsion free. Let $L$ be an immediate extension of the real closure $C^\rc$ of $C$. Then
\[
	\tp(M/C) \vdash \tp(M/L) .
\]
\end{corollary}

\begin{proof}
The fact that $\tp(M/C) \vdash \tp(M/Ck_{C^\rc}\Gamma_{C^\rc})$ is immediate. By the previous corollary,  
$\tp(M/Ck_L\Gamma_L) \vdash \tp(M/L)$. The conclusion follows as $k_{C^\rc}=k_L$ and $\Gamma_{C^\rc}=\Gamma_L$.
\end{proof} 

We now move on to proving our stronger domination result: over a maximal base, the ordered 
valued field is dominated by the $k$-internal sorts
over its value group. Those readers familiar with \cite[12.15]{HHM2} may safely skip to the final claim of the proof 
of Theorem \ref{real-prop-12.15} below.  However, as the proof in \cite{HHM2} incorrectly defines the parameters $e_i$, tacitly 
uses the uniqueness of non-forking extensions in a stable theory (not available to us), and, rather than using \cite[12.11]{HHM2} 
directly, uses the following equivalent formulation not explicitly stated in \cite{HHM2}, we hesitate to refer new readers to the proof in 
\cite{HHM2}, and have instead reproduced it here.  

\begin{corollary}\label{prop-12-11-alternate}
Let $C$, $L$, and $M$ be substructures of a large algebraically (respectively real) closed valued field.  Assume that $C$ is a 
maximal substructure of both $L$ and $M$ with $\Gamma_L\cap \Gamma_M =\Gamma_C$, and $k_L$, 
$k_M$  linearly disjoint over $k_C$.   If $L' \models \tp(L/C\Gamma_Mk_M)$ then $L' \models \tp(L/M)$.  
Furthermore, $\Gamma_{C(L,M)}$ is the group generated by $\Gamma_L$ and $\Gamma_M$, and $k_{C(L,M)}$
 is the field generated by $k_L$ and $k_M$.
\end{corollary}

\begin{proof}
Let $\sigma: L\to L'$ witness $L' \models \tp(L/C\Gamma_Mk_M )$.  We wish to extend $\sigma$ to an isomorphism 
$\widetilde \sigma:C(L,M)\to C(L',M)$ which is the identity on $M$.  Suppose that $M = C(m)$ where $m$ is a possibly infinite 
tuple and that $\sigma(m)=\widetilde m$.  We need to find $\tau$ such that $\tau: \widetilde m \mapsto m$ while fixing $L$, 
$k_M$ and $\Gamma_M$.  Letting $\widetilde M = C(\widetilde m)$ we see that $M \models \tp (\widetilde M /C\Gamma_Mk_M )$.  
We apply \cite[12.11]{HHM2} (resp. Theorem \ref{real-prop-12.11}) with the roles of $L$ and $M$ reversed.  This gives us both 
the desired $\tau$ and proves the desired statements about the value group and residue field of $C(L,M)$.  Then we 
define $\widetilde \sigma = \tau \circ \sigma$.
\end{proof}

Below, $\R$ is a real closed valued field.
\begin{theorem}\label{real-prop-12.15}
Let $C$ be a maximal real closed valued field. Let $L$, $M$ be real closed valued fields with $C\subseteq L\cap M$ and  
$\Gamma_L\subseteq\Gamma_M$.  Assume that $\kInt_{\Gamma_L}^L$ is algebraically independent from 
$\kInt_{\Gamma_L}^M$ over $C\Gamma_L$. Suppose $\sigma$ is an automorphism of $\R$ fixing  
$C\Gamma_L \kInt_{\Gamma_L}^M $ and such that $\sigma(L)=L'$. 
Then there is an automorphism of $\R$ which restricts to $\sigma$ on $L$, fixes $M$ pointwise, and maps  $C(L,M)$ to $C(L',M)$.
\end{theorem}

\begin{proof}

In outline, we begin, as in \cite[Proposition 12.15]{HHM2}, by perturbing the valuation to a finer one, $v'$, which satisfies the 
hypothesis that $\Gamma_{(L,v')}\cap \Gamma_{(M,v')} =\Gamma_{(C,v')}$. We can then apply Corollary~\ref{prop-12-11-alternate}
to extend $\sigma |_{L}$ to a valued field isomorphism from $C(L,M)=N$ to $C(L',M)=N'$ which extends the identity on $M$.
Finally we show that $\sigma$ also preserves the ordering on the fields. This last step requires some details
of the construction of the perturbation of the valuation, so it is worth repeating (and correcting) the proof
from \cite{HHM2} here. We use the language of places for this construction; information about places can be found
in \cite[Chapter VI]{ZS}.

The assumption that $L$ and $L'$ satisfy the same type over $C\Gamma_L \kInt_{\Gamma_L}^M $ allows us to assume that $\sigma$ fixes 
$C\Gamma_M \kInt_{\Gamma_L}^M $.  For suppose that there is some $\mu\in\Gamma_M$ with $\sigma(\mu)=\mu'$.  We wish to show there 
is $\tau: \mu' \mapsto \mu$ fixing $L' \kInt_{\Gamma_L}^M $ so that we may replace $\sigma$ with $\tau\circ \sigma$.  By stable embeddedness 
of $\Gamma$, we must show that $\mu$ and $\mu'$ realize the same type over $\Gamma(L' \kInt_{\Gamma_L}^M) $, so it suffices to show 
$\Gamma(L' \kInt_{\Gamma_L}^M) = \Gamma_L$.  Suppose that there is an $L'$-definable function $f$ from $\kInt_{\Gamma_L}^M$ to $\Gamma$.    
By the orthogonality of $k$ and $\Gamma$, for every $\lambda\in \Gamma_L$, $f$ takes only finitely many values on $RV_\lambda(M)$, and thus 
these values are algebraic over $L'$ and hence in $\Gamma_{L'}= \Gamma_L$.

Choose $a_1,\ldots, a_r$ from $L$ and $e_1,\ldots,e_r$ from $M$ such that, for each $1\le i\le r$, $v(a_i)=v(e_i)$
and $\{ v(a_i)\}$ forms a $\bQ$-basis for $\Gamma_L$ modulo $\Gamma_C$. Choose $b_1,\ldots,b_s$ from
$L$ such that $\{ \res(b_1),\ldots,\res(b_s)\}$ is a transcendence basis for $k_L$ over $k_C$. By Proposition~\ref{real-prop-12.10},
the elements
\[
	\res(a_1/e_1),\ldots,\res(a_r/e_r),\res(b_1),\ldots,\res(b_s)
\]
are algebraically independent over $k_M$. For $0\le j \le r$, let
\[
	R^{(j)} = \dcl(k_M ,\res(a_1/e_1),\ldots,\res(a_j/e_j),\res(b_1),\ldots,\res(b_s)).
\]
In particular, 
\begin{align*}
	R^{(0)} & = \dcl(k_M ,\res(b_1),\ldots,\res(b_s)) = \dcl(k_M ,k_L ) \quad \text{and} \\
	R^{(r)} & = \dcl(k_M ,\res(a_1/e_1),\ldots,\res(a_r/e_r), k_L ) .
\end{align*}
For each $0\le j\le r-1$, choose a place $p^{(j)}: R^{(j+1)} \to R^{(j)}$ fixing $R^{(j)}$ and such that 
$p^{(j)}(\res(a_{j+1}/e_{j+1})) =0$, which is possible by the algebraic independence of $\res(a_1 /e_1), \dots ,\res(a_r /e_r )$ over $k_M$. Also choose a place $p^*:k(N) \to R^{(r)}$ fixing $R^{(r)}$.
Write $p_v:\dcl(N)\to k(N)$ for the place corresponding to our given valuation $v$. Define
$p_{v'}: \dcl(N)\to R^{(0)}$ to be the composition
\[
	p_{v'} = p^{(0)}\circ \cdots \circ p^{(r-1)}\circ p^*\circ p_{v} .
\]
Let $v'$ be a valuation associated to the place $p_{v'}$. 
Notice that all the places $p^{(j)}$ and  $p^*$ are the identity on $k_M$, so we may identify 
$(M,v)$ and $(M,v')$, including identifying the value groups $\Gamma_M$ and $\Gamma_{(M,v')}$.
Similarly, the places are all the identity on $k_L$, so the value groups $\Gamma_L$ and $\Gamma_{(L,v')}$
are isomorphic, but we shall see that we cannot simultaneously identify $\Gamma_M$ with $\Gamma_{(M,v')}$ and
$\Gamma_L$ with $\Gamma_{(L,v')}$. Also notice that we cannot expect $v'$ to be convex with respect
to the ordering.

We now have two valuations $v$ and $v'$ on $N$. If $x\in M\subseteq N$, then $v(x)=v'(x)$, and if
$x,y\in L\subseteq N$ then $v(x) \le v(y)$ implies $v'(x)\le v'(y)$. Furthermore, the construction has
ensured that for any $x\in M$ with $v(x)>0$, 
\[
	0< v'(a_1/e_1) \ll \cdots \ll v'(a_r/e_r) \ll v'(x),
\]
where $\gamma \ll \delta$ means that $n\gamma < \delta$ for any $n\in\bN$ (and hence $\Gamma_{(L,v')}\not=\Gamma_L$). Let $\Delta$ be the 
subgroup of $\Gamma_{(N,v')}$ generated by $v'(a_1/e_1),\ldots,v'(a_r/e_r)$. Then $\Delta$
is a convex subgroup of $\Gamma_{(N,v')}$ and $\Gamma_{(N,v')} = \Delta\oplus\Gamma(N)$ (where the right-hand group is ordered
lexicographically). 

To see that $\Gamma_{(L,v')}\cap \Gamma_{(M,v')} = \Gamma_{(C,v')}$, 
let $m\in M$ and $\ell\in L$ be such that $v'(m)=v'(\ell)$.  Set $v'(\frac{a_i}{e_i})=\delta_i$ and $v'(e_i )=\epsilon_i$.  As $(v(a_i))$ generates
$\Gamma_L$ over $\Gamma_C$, and $\Gamma_L$ and $\Gamma_{(L,v')}$ are isomorphic, 
$$v'(\ell)=\sum_{i=1}^{r}p_i v'(a_i ) + \gamma = 
	\sum_{i=1}^{r} p_i \delta_i + \sum_{i=1}^{r} p_i \epsilon_i + \gamma ,
$$ 
where $p_i \in \mathbb{Q}$ and $\gamma \in \Gamma_C$. The set 
$$\{ \delta_1 , \dots ,\delta_r , \epsilon_1 ,\dots ,\epsilon_t \}
$$ 
is algebraically independent over $\Gamma_C$ since $\Gamma_{(N,v')}=\Delta \oplus \Gamma_N$.  Next, note that 
since $v'(e_i) = v(e_i )$, $\{ v'(e_i ) \}$ forms a $\mathbb{Q}$-basis of $\Gamma_L \subseteq \Gamma_M = \Gamma_{(M,v')}$ 
over $\Gamma_C$.  Let $\mu_1 , \dots ,\mu_t$ be such that $\{ \epsilon_i \} \cup \{ \mu_j \}$ forms a $\mathbb{Q}$-basis of 
$\Gamma_M$ over $\Gamma_C$.  Then 
$$ v'(m)  = \sum_{i=1}^{r}p'_i \epsilon_i + \sum_{i=1}^{t}q_i \mu_i + \gamma',$$ 
where $q_i \in \mathbb{Q}$ and $\gamma' \in \Gamma_C$.  It follows that each $p_i =p'_i =0$ and each $q_i = 0$, 
hence $v'(\ell)=v'(m) \in \Gamma_C$.

Next we must check that $k_{v'}(L)$ and $k_{v'}(M)$ are linearly disjoint. But note that the place 
$p^{(0)}\circ \cdots \circ p^{(r-1)}\circ p^*:k(N)\to \dcl(k(M),k(L))$ is the identity on $k_M$ and 
$k_L$ and thus also on their compositum.  Thus $k_L$ and $k_M$ being linearly disjoint over $k_C$ implies linear disjointness of 
$k_{(L,v')}$ and $k_{(M,v')}$ over $k_{(C,v')}$.

Hence we can apply Corollary~\ref{prop-12-11-alternate} to deduce that the isomorphism $\sigma |_{L}$ extends 
to a valued field isomorphism $\sigma'\colon (N,v') \to (N',v')$ which is the identity on $M$. As $v'$ is a refinement of
$v$, $\sigma'$ is also an isomorphism of $(N,v)$. 

It remains to show that $\sigma'$ preserves the ordering on $N$. Assume to the contrary, and let $n$ be minimal with
$\sum_{i=1}^{n}\ell_i m_i  >0$ and $\sigma'(\sum_{i=1}^{n} \ell_i m_i)<0$, where $m_i \in M$ and $\ell_i \in L$.  
Since $\Gamma_{(LM)}=\Gamma_M$, we may assume that $v(\sum_{i=1}^{n} \ell_i m_i)=0$.  Let $U = C^n\cdot m$ 
where $m=(m_1, \dots, m_n)$.  Modifying $m$ if needed we may assume that $m$ forms a separated basis for $U$ over $C$ with respect 
to $v'$, and hence, by Fact~\ref{factseparated}, is also separated over $L$ with respect 
to $v'$.   In fact, $m$ is also separated over $L$ with respect to $v$: for consider an element 
$x=\tilde{\ell}_1 m_1 + \dots + \tilde{\ell}_n m_n$ of $L^n\cdot m$.  Since $\Gamma_{(N,v')}=\Delta \oplus  \Gamma_N$ and by the 
construction of $v'$, $v'(x)=\delta + v(x)$  for some $\delta\in \Delta$. On the other hand, 
$$v'(x)= \textrm{min}\{v'(\tilde{\ell}_1 m_1), \dots, v'(\tilde{\ell}_n m_n)\} = 
	\textrm{min}\{\delta_1 + v(\tilde{\ell}_1 m_1), \dots, \delta_n+v(\tilde{\ell}_n m_n)\}.
$$  
Suppose the minimum is achieved at $j$. Then $v(x)=v(\tilde{\ell}_j m_j)$. To show that $v(\tilde{\ell}_j m_j) \le v(\tilde{\ell}_i m_i)$ for all $i$, suppose
for a contradiction that it is not.  Then $\delta_{j} + v(\tilde{\ell}_{j} m_{j}) \le \delta_i + v(\tilde{\ell}_{i} m_{i})$  implies 
$$0< v(\tilde{\ell}_j m_j)-v(\tilde{\ell}_i m_i ) \leq \delta_i - \delta_j ,
$$ 
a contradiction with the convexity of $\Delta$.

Since $m$ forms a separated basis of $L^n \cdot m$ over $L$, we have $$v(a)=\textrm{min}\{v(\ell_1 m_1), \dots, v(\ell_n m_n )\} .$$  
By the minimality of $n$, 
$v(\ell_1m_1)=\dots = v(\ell_n m_n)$, since if, say, $v(\ell_1 m_1 )>v(a)$, then subtracting $\ell_1 m_1$ from $a$ would not change the sign of $a$, 
nor would subtracting $\sigma'(\ell_1m_1)$ from $\sigma'(a)$.  Thus for each $i$, $v(\ell_i m_i) =0$, so 
$$v(m_i)= - v(\ell_i) = \lambda_i \in \Gamma_L .$$ 

This implies $\ell_1 \ball{op}{\lambda_1}{m_1} +\dots + \ell_n \ball{op}{\lambda_n}{m_n} > 0$, since an element thereof is of the form 
$$ \ell_1(m_1+d_1) + \dots + \ell_n (m_n+d_n) = a + (\ell_1 d_1 + \dots + \ell_n d_n)$$ for some $d_1,\ldots,d_n$ with $v(d_i) > \lambda_i$.
As $v(\ell_1 d_1 + \dots + \ell_n d_n)\ge \min\{v(\ell_id_i)\}>0$, we see that adding it to $a$ cannot change the sign of $a$. Thus 
\[
x_1 \ball{op}{\lambda_1}{m_1} +\dots + x_n \ball{op}{\lambda_n}{m_n} > 0
\]
is a formula in $\tp(\ell_1,\dots , l_n/\kInt_{\Gamma_L}^{M})$, 
which $\sigma'$ preserves.  So $\sigma'(a)>0$.

\end{proof}

As for Theorem~\ref{real-prop-12.11}, we can restate the theorem in terms of domination.

\begin{corollary}\label{dom-by-resfield-over-valuegp}
Let $C\subseteq L$ be elementary substructures of $\R$ with $C$ maximal. Then $\tp(L/C)$ is dominated over its value group 
by the $k$-internal sorts. \end{corollary}

\begin{proof}
Let $M\supseteq C\Gamma_L$ be another substructure of $\R$ and assume that 
\[ \kInt_{\Gamma_L}^M \thind_{\kInt_{\Gamma_L}^C} \kInt_{\Gamma_L}^L.
\] 
We need to show that  $\tp(L/ C\Gamma_L\kInt_{\Gamma_L}^M) \vdash \tp(L/C\Gamma_LM)$. That is, as in 
Fact~\ref{domination-automorphism-definition}, for any automorphism $\sigma$ of $\R$ fixing $C\Gamma_L\kInt_{\Gamma_L}^M$, 
there is an automorphism agreeing with $\sigma$ on $L$ and fixing $M$. This is the conclusion of Theorem~\ref{real-prop-12.15}
(for $M^\rc$ and hence also for $M$). The hypothesis holds by Proposition~\ref{real-prop-12.10}, as 
$\kInt_{\Gamma_C}^{M^\rc} = \kInt_{\Gamma_C}^M$.
\end{proof}

\begin{corollary}\label{not-alg-closed}
Assume $C$, $L$, $M$ are as in Theorem~\ref{real-prop-12.15}, except not necessarily real closed, but with $k_L$ a regular 
extension of $k_C$ and $\Gamma_L/\Gamma_C$ torsion free. Then 
\[
	\tp(L/C\Gamma_L\kInt_{\Gamma_L}^M) \vdash \tp(L/M).
\]
\end{corollary}

\begin{proof}
We may assume $M$ is maximal and real-closed, so that $\bar{C}$, an immediate maximal extension of $C^{rc}$, may be constructed 
within $M$.  As an aside, note that $\bar{C}$ is in fact real closed.  For otherwise there would be an element $c$ algebraic over $\bar{C}$.  
But $c$ could not generate a ramified extension of $\bar{C}$, as it would necessarily give a new element in the divisible hull of 
$\Gamma_{\bar{C}}$, which already is divisible.  Likewise it cannot be a residual extension, as $k(C)$ is already real closed.  
Nor can it be immediate, as $\bar{C}$ is maximal.  

Without loss of generality, $L$ is $C(\ell)$ for a finite tuple $\ell$.  Let $\bar{L}$ be $\bar{C}(\ell)$ and let $\bar{L}'$ be $\bar C(\ell')$ (where $\ell'$ is 
$\sigma(\ell)$).  It suffices to show that $\bar{L}, \bar{L}', M$ and $\bar{C}$ satisfy the hypotheses of Theorem~\ref{real-prop-12.15}, for
then the conclusion of the theorem is that $\sigma$ extends by the identity on $M$, and hence that $L$ and $L'$ realize the same type 
over $M$ as desired.

Applying Theorem~\ref{real-prop-12.11}, we see that $\Gamma_{\bar{L}}$ is generated by $\Gamma_{L}$ and $\Gamma_{\bar{C}}$, and that
$k_{\bar{L}}$  is generated by $k_{L}$ and $k_{\bar{C}}$. Since $\Gamma_{L}$ and $\Gamma_{\bar{C}}$ are both contained in 
$\Gamma_M$, we have established $\Gamma(\bar{L})\subseteq \Gamma(M)$. Furthermore, there is  a subtuple  $(\ell_{n_i})$ of $\ell$ in $L$ such 
that $(v(\ell_{n_i}))$ generates both $\Gamma_L$ over $\Gamma_C$ and $\Gamma_{\bar{L}}$ over $\Gamma_{\bar{C}}$. Also there is
a subtuple $(\ell_{n_j})$ of $\ell$ such that $(\res(\ell_{n_j}))$ generates both $k_L$ over $k_C$ and  $k_{\bar{L}}$ over
$k_{\bar{C}}$. Taking any $e_i$ in $M$ with $v(e_i) = v(\ell_{n_i})$, we see that the hypothesis 
$\kInt_{\Gamma_L}^L \thind_{C\Gamma_L}\Gamma_L\kInt_{\Gamma_L^M}$ implies by Proposition~\ref{real-prop-12.10} that the sequence
$\{\res(\ell_{n_i}/e_i), \res(\ell_{n_j})\}_{i,j}$ is algebraically independent over $k_M$, and the existence of such elements implies that
$\kInt_{\Gamma_{\bar{L}}}^{\bar{L}} \thind_{\bar{C}\Gamma_{\bar{L}}}\kInt_{\Gamma_{\bar{L}}}^M$.

Finally, we need to verify that $\bar{L}$ and $\bar{L'}$ satisfy the same type over $\bar{C}\kInt_{\Gamma_{\bar{L}}}^M$.
We have that $\ell$ and $\ell'$ satisfy the same type over $C\Gamma_L\kInt_{\Gamma_L}^M$.  Choose a tuple $m$  so that the 
$\res(m_i)$ generate $k(M)$ over $k(C)$ (and each $m_i$ generates a residual extension). As in Corollary~\ref{dom-by-valuegp-resfield}, 
the hypothesis that 
$k_L$ is a regular extension of $k(C)$  means that the independence of $k_L$ and $k_M$ over $C$ implies linear disjointness.  Similarly, 
as $\Gamma_L/\Gamma_C$ is torsion free, and $\Gamma_{\bar{C}(m)}=\Gamma_{\bar C}$, we have that 
$\Gamma_L\cap\Gamma_{\bar{C}(m)}=\Gamma_C$.  Thus we may assume that $\sigma(\bar{C}(m))=\bar{C}(m)$. Now we may
apply Theorem \ref{real-prop-12.11} to $L=C(\ell)$ and $\bar{C}(m)$ and conclude that we have 
an automorphism mapping $L$ to $L'$ fixing $\bar C$ and $\kInt_{\Gamma_{\bar L}}^M \subseteq \bar{C}(l,m)$, as required.

\end{proof}


\section{Forking and \th-forking}
Theorem 2.5, especially as expressed in the form of Corollary 2.8, has the pleasing consequence that forking and \th-forking over a maximal base
are controlled by the value group and residue field. 

We begin by recalling the definitions and a few basic properties of forking and \th-forking.

\begin{definition}
A formula $\varphi(x,b)$ {\em divides} over $C$ if there is a sequence $(b_i)_{i<\omega}$ in $\tp(b/C)$ with $b=b_0$ and 
$\{\varphi(x,b_i)\}$ $m$-inconsistent. A formula $\varphi(x,b)$ {\em \th-divides} over $C$ if there is $d$ such that 
$\{\varphi(x,\tilde b): \tilde b \models \tp(b/Cd)\}$ is infinite and $m$-inconsistent. A formula {\em (\th-)forks} over $C$ if it implies a disjunction 
of formulas which (\th-)divide over $C$.  We say that $\tp(a/bC)$ {\em (\th-)forks} over $C$ if it contains a formula which (\th-)forks over $C$.  If 
$\tp(a/bC)$ does not (\th-)fork over $C$, we say $a$ is {\em (\th-)independent } from $b$ over $C$ and write this as $a \ind_C b$ 
(respectively $a \thind_C b$). Clearly, \th-forking implies forking.
\end{definition}

The difference between forking and dividing can sometimes be an issue, but in a large class of theories, including weakly o-minimal 
theories and algebraically closed valued fields, forking and dividing are the same \cite[the remarks preceding Proposition 2.6 together 
with Corollary 5.5]{CS}. 

Forking is not transitive in a non-simple theory, but the following partial left transitivity (sometimes called the pairs lemma) holds in all theories.  Often stated for dividing, it can be seen to hold for forking as well.  The corresponding property of \th-forking also holds in all theories.

\begin{fact}\label{pairs}
 In any theory 
  \begin{enumerate}[i)]
    \item \cite[Lemma 2.1.6]{O} if $a\thind_A c$ and $b\thind_{Aa} c$ then $ab\thind_A c$; and
    \item \cite[Lemma 1.5]{S} if $a\ind_A c$ and $b\ind_{Aa} c$ then $ab\ind_A c$. 
  \end{enumerate}
\end{fact}

\begin{lemma} \label{independence_in_k_and_Gamma}  In either ACVF or RCVF,
  \begin{enumerate}[i)]
    \item $k(Ca)\Gamma(Ca)\ind_C b$ if and only if $k(Ca)\Gamma(Ca)\ind_C k(Cb)\Gamma(Cb)$; and
    \item $k(Ca)\Gamma(Ca)\thind_C b$ if and only if $k(Ca)\Gamma(Ca)\thind_C k(Cb)\Gamma(Cb)$.
  \end{enumerate}
\end{lemma}

\begin{proof}
In both i) and ii) the left to right implication is clear, as  $k(Ca)\Gamma(Ca)\ind_C b$ implies $k(Ca)\Gamma(Ca)\ind_C \acl(Cb)$ and 
$\acl(Cb)$ contains $k(Cb)\Gamma(Cb)$ (and similarly for \th-independence).
  
Since $k$ and $\Gamma$ are orthogonal, $k(Ca)\Gamma(Ca)\ind_C b$ if and only if $k(Ca) \ind_C b$ and 
$\Gamma(Ca)\ind_C b$.  Furthermore, one has $k(Ca)\Gamma(Ca)\ind_C k(Cb)\Gamma(Cb)$ if and only if 
$k(Ca) \ind_C k(Cb)$ and $\Gamma(Ca)\ind_C \Gamma(Cb)$.  Thus we need to show that $k(Ca) \ind_C k(Cb)$ implies 
$k(Ca) \ind_C b$ and also that $\Gamma(Ca)\ind_C \Gamma(Cb)$ implies $\Gamma(Ca)\ind_C b$.

Assume that $k(Ca) \nind_C b$.  Then, as forking is dividing, there is a formula $\varphi(x,a) \in \tp(k(Ca)/bC)$, and there is  
$m<\omega$, and $b=b_0, b_1, b_2, \dots$ satisfying $\tp(b/C)$ such that $\{\varphi(x,b_i)|i<\omega\}$ is $m$-inconsistent.  By stable 
embeddedness of $k$, the subset of $k^n$ defined by $\varphi(x, b)$ is also defined by $\psi(x, e)$ for $e$ a tuple from the residue 
field.  As $k$ eliminates imaginaries, we may assume $e$ is the canonical parameter for the set defined by $\varphi(x,b)$.  Thus, as this 
set is definable over $Cb$, $e$ is contained in $k(Cb)$.  Suppose that $\sigma_i$ is an automorphism of $\mathcal{R}$ that maps $b$ to $b_i$ 
and fixes $C$.  Then, letting $e_i=\sigma_i(e)$, one sees that $\{\psi(x,e_i) | i<\omega\}$ is $m$-inconsistent and witnesses $k(Ca) \nind_C k(Cb)$.

The proof that $\Gamma(Ca)\ind_C \Gamma(Cb)$ implies $\Gamma(Ca)\ind_C b$ is similar.

Now suppose $k(Ca) \nthind_C b$.  We first consider the case in which $\tp(k(Ca)/Cb)$ contains a formula that \th-divides over $C$.

\begin{unnumberedclaim}
  Suppose $\varphi$ is a formula over $C$ such that $\varphi(x, b)$ defines a subset of $k^n$ and \th-divides over $C$. Then there 
  is $\psi(x, e)$, defining the same subset of $k^n$, which also \th-divides over $C$, and  where $e\in k(Cb)$.
\end{unnumberedclaim}

\begin{proof of claim}
 We have $d$ such that $\{\varphi(x,\tilde b)\colon \tilde b \models \tp(b/Cd)\}$ is $m$-inconsistent and infinite.  As above, we can 
 replace $\varphi(x, b)$ with $\psi(x, e)$ defining the same set and with $e$ in $\dcl(Cb)$.  
 Thus $\psi(x, e)$ \th-divides over $C$.
\end{proof of claim}

Now suppose $\tp(k(Ca)/Cb)$ contains $\varphi(x, b)$ which implies $\bigvee \varphi_i(x, b_i)$, where each 
$\varphi_i(x, b_i)$ \th-divides over $C$.  By stable embeddedness of $k$, we may replace $\varphi(x, b)$ with $\psi(x,e)$ 
defining the same set and with $e\in k(Cb)$.  By the claim, we may replace each $\varphi_i(x, b_i)$ with a $\psi_i(x, e_i)$ defining 
the same set and with $e_i\in k(Cb_i)$.  Since $\varphi$ and $\psi$ define the same set, as do $\varphi_i$ and $\psi_i$, we have $\psi(x,e)$ 
implies $\bigvee \psi_i(x, e_i)$.   Thus $k(Ca) \nthind_C k(Cb)$. 

The proof that $\Gamma(Ca)\thind_C \Gamma(Cb)$ implies $\Gamma(Ca)\thind_C b$ is similar.
\end{proof}

\begin{theorem}\label{forking-over-maximal}
Let $C$ be a substructure of a model $\R$ of either RCVF or ACVF, and let $a,b\in \R$. Assume that $C$ is a model (or just that $k(Ca)$ is a regular extension of $k_C$ and $\Gamma(Ca)/\Gamma_C$ is torsion free) and maximal.  Then 
\begin{enumerate}[i)]
  \item $a \thind_C b$ if and only if $k(Ca)\Gamma(Ca)\thind_C k(Cb)\Gamma(Cb)$;
  \item $a \ind_C b$ if and only if $k(Ca)\Gamma(Ca)\ind_C k(Cb)\Gamma(Cb)$.
\end{enumerate}
\end{theorem}

\begin{proof}

In both i) and ii), the left to right direction is clear.  Now assume that $k(Ca)\Gamma(Ca)\thind_C k(Cb)\Gamma(Cb)$,  that is, 
$k(Ca)\thind_C k(Cb)$ and 
$\Gamma(Ca)\thind_C \Gamma(Cb)$.  The former (together with $C$ being a model or our regularity assumption) implies that $k(Ca)$ and 
$k(Cb)$ are linearly disjoint over $C$ while the latter (together with $C$ being a model or our torsion free assumption) implies that 
$\Gamma(Ca)\cap\Gamma(Cb)=\Gamma(C)$.  Thus, we may apply Corollary \ref{prop-12-11-alternate} in the following fashion to see 
that $a \ind_{Ck(Ca)\Gamma(Ca)} b$: choose an indiscernible sequence $b=b_0, b_1, \dots$ in $\tp(b/Ck(Ca)\Gamma(Ca))$.  Let 
$p(x, Ck(Ca)\Gamma(Ca)b)=\tp(a/Ck(Ca)\Gamma(Ca)b)$.  
Since, by the corollary, a partial isomorphism mapping $b$ to $b_i$ and fixing $Ck(Ca)\Gamma(Ca)$ may be extended to one fixing 
$a$ as well, we have that 
$a\models \bigcap_i p(x,Ck(Ca)\Gamma(Ca)b_i)$ and thus $\tp(a/Ck(Ca)\Gamma(Ca)b)$ does not divide (and hence does not fork) 
over $Ck(Ca)\Gamma(Ca)$.

By Lemma \ref{independence_in_k_and_Gamma}, $k(Ca)\Gamma(Ca)\thind_C b$ if and only if 
$k(Ca)\Gamma(Ca)\thind_C k(Cb)\Gamma(Cb)$.  Likewise, $k(Ca)\Gamma(Ca)\ind_C b$ if and 
only if $k(Ca)\Gamma(Ca)\ind_C k(Cb)\Gamma(Cb)$.    Using Fact \ref{pairs}, $a \ind_{Ck(Ca)\Gamma(Ca)} b$ and 
$k(Ca)\Gamma(Ca) \ind_C b$ imply $ak(Ca)\Gamma(Ca)\ind_C b$.  
Since $k(Ca)\Gamma(Ca)$ is in $\acl(Ca)$, this is equivalent to $a \ind_C b$.  

Likewise (recalling independence implies \th-independence), 
$a \ind_{Ck(Ca)\Gamma(Ca)} b$ and $k(Ca)\Gamma(Ca) \thind_C b$ together imply $a\thind_C b$.
 \end{proof}

As $k\times\Gamma$ is rosy (and hence \th-forking is symmetric), we have the following corollary.
\begin{corollary}
 $a\thind_C b$ if and only if $b\thind_C a$ when $C$ is a maximal model.  
\end{corollary}
\noindent Of course, \th-forking is symmetric in general if and only if the theory is rosy, so one cannot expect this to hold without $C$ being a maximal model.


\section{Geometric sorts in a $T$-convex theory}

We now turn to the problem of extending the above domination results to the geometric sorts. This requires developing
the work on prime resolutions from \cite{HHM2} to apply to the case of ordered valued fields. Much of the structure 
below comes from \cite[Chapter 11]{HHM2}, however there are enough differences in the ordered case to make it 
worthwhile to reproduce it here.

\subsection{Prime resolutions}
Our goal in this subsection is to prove the existence of prime resolutions of definable sets in the geometric language
in a power-bounded $T$-convex theory. Recall that the geometric sorts in a valued field $K$ are a collection
of definable $V_K$-submodules of $K^n$ and their torsors. In a pure algebraically closed valued field or
real closed valued field,  the theory eliminates imaginaries with respect to these sorts \cite{HHM1,M}. This 
collection of sorts does not suffice to eliminate imaginaries in a richer language which includes a function symbol
for the restricted exponential function \cite{HHM3}. Nevertheless, it is still of interest to understand domination in
the context of this sorted language.

The geometric sorts in a valued field can be identified with coset spaces of matrix groups over the field, as 
follows (full details of this identification can be found in \cite[Chapter 7]{HHM2}). Let $B_n(K)$ ($B_n (V_K )$ and 
$B_n (k_K)$ respectively) be the (multiplicative) group
of upper triangular invertible $n \times n$ matrices over $K$ ($V_K$ and $k_K$ respectively).  Further, let $B_{n,m}(k_K)$ 
be the set of elements of $B_n(k_K)$ whose $m$th column has a $1$ as the $m$th entry and $0$ elsewhere. 
Let $B_{n,m}(V_K)$ be the set of matrices in $B_n(V_K)$ which reduce modulo $\m$ componentwise
to elements of  $B_{n,m}(k_K)$. The sort $S_n$ of $V_K$-submodules of $K^n$ can be identified with the 
set of codes for left cosets of $B_n(V_K)$ in $B_n(K)$ and the sort $T_n$ of torsors of elements of $S_n$
can be identified with a set of codes for elements of $\bigcup_{m=1}^n B_n(K) / B_{n,m}(V_K)$.

We set $\mathcal{G}=\bigcup_{n=1}^{\infty }\big(\mathcal{S}_n \cup \mathcal{T}_n  \big)$.

\begin{definition}
Let $A$ be a substructure of $\Eq{\R}$. A {\em pre-resolution} of $A$ is a  substructure $D$ of $\Eq{\R}$ such that 
$A\subseteq\acl(D\cap \R)$. A {\em resolution} of $A$ is a substructure $D$ of $\R$ which is algebraically closed in 
$\R$ and is such that $A\subseteq \dcl{D}$. The resolution $D$ is prime over $A$ if it embeds over $A$ into any other resolution.
\end{definition}

The proof of existence of prime resolutions goes via properties of opaque equivalence relations. We summarize 
this briefly from \cite{HHM2}.

\begin{definition}\cite[Definition 11.5]{HHM2}
\begin{enumerate}[i)]
	\item Let $E$ be a $C$-definable equivalence relation on a $C$-definable set $D$. We say that $E$ is {\em opaque}
if, for each $C$-definable $F \subseteq D$, $F=F'\cup F''$, where $F'$ is a union of $E$-equivalence classes and $F''$ is contained in 
a finite union of $E$-equivalence classes.
	\item We say that $\tp(a/C)$ is {\em opaquely layered} if there are sequences $a_1,\ldots,a_N$ of (imaginary) 
elements and $E_1,\ldots,E_N$ of opaque equivalence relations such that each $a_i$ is an $E_i$-equivalence class, each $E_i$ is defined 
over $C\cup\{a_j:j<i\}$, and $\dcl(Ca)=\dcl(C, a_1,\ldots,a_N)$.
	\item The $C$-definable equivalence relation $E$ on $D$  is {\em opaquely layered} over $C$ if for every $d\in D$,
	$\tp(d/E)$ is opaquely layered over $C$.
	\item Let $G$ be a $C$-definable group and $F$ a $C$-definable subgroup. Then $G/F$ is opaquely layered (opaque)
	over $C$ if  the equivalence relation $xF=yF$ is opaquely layered (opaque) over $C$.
\end{enumerate}
\end{definition}

\begin{proposition}\label{pre-resolution-exists}\cite[Lemma 11.7]{HHM2}
Let $C\subseteq \R$ and suppose that $a$ is a finite tuple of imaginaries such that $\tp(a/C)$ is opaquely layered, witnessed
by $a_0,\ldots,a_N$ and $E_0\ldots,E_N$. Then $\dcl (Ca)$ has a pre-resolution $D$ which is atomic over 
$C\langle a\rangle$. If $D'$ is another pre-resolution that contains $C$ and an element from each $E_i$-class, 
then $D$ embeds elementarily into $D'$. 
\end{proposition}

\noindent
In the situation of the present paper, where the elements of $\Eq{\R}$ are equivalence classes in the home sort $\R$, the 
statement can be strengthened to say that $D$ is a resolution, not just a pre-resolution.

We now apply these concepts to build prime resolutions for the geometric sorts in $T$-convex theories. 

\begin{proposition}\label{opaque-groups}
Let $K$ be a model of a power-bounded $T$-convex theory. 
\begin{enumerate}[i)]
	\item The additive groups $K/V_K$ and $V/\m_K$ are opaque, as is the multiplicative group $K^\times/(1+\m_K)$.
	\item The multiplicative group $K^\times / V^\times_K$ is 
	opaquely layered.
	\item The groups $B_n(K)/B_n(V_K)$ and $B_n(K)/B_{n,m}(V_K)$ are opaquely layered.
\end{enumerate}
\end{proposition}

\begin{proof} Part i) is an immediate consequence of the property arising from Fact~\ref{definablesets}, that the definable subsets of $K$ 
have finitely many convex components 
and that the equivalence classes of each of the given quotients are convex subsets of $K$.   

To prove part ii), let $E$ be the equivalence relation $$xEy \Longleftrightarrow v(x)=v(y)$$ on $K^{\times}$.  Then the fact that the type of the 
$E$-equivalence class of $b\in K^{\times}$ is opaquely layered is witnessed by the equivalence relations
\begin{align*}
	x E_1 y & \Longleftrightarrow x/y > 0; \\
	x E_2 y & \Longleftrightarrow x E_1 y \text{ and } v(x) = v(y)
\end{align*}
and by the sequence of imaginary elements $b/E_1 , b/E_2$.

iii) This is exactly the same as the proof of the corresponding statement in \cite[Lemma 11.13]{HHM2}.
\end{proof}

\begin{theorem}\label{resolution-exists}
Let $K$ be a power-bounded $T$-convex structure, and let $C\subseteq K$.
Let $e$ be a finite set of imaginaries
from $\mathcal{G}$. Then $\dcl (Ce)$ admits a resolution $D$ which is minimal, prime and atomic over $Ce$. Up to isomorphism 
over $Ce$, $D$ is the unique prime resolution of $\dcl (Ce)$. Furthermore, $k_D =k(Ce)$, $\Gamma_D=\Gamma(Ce)$
and $\kInt_{\Gamma(D)}^D = \kInt_{\Gamma(Ce)}^{\acl(Ce)}$.
\end{theorem}

\begin{proof}
The proof is essentially the same as that of \cite[Theorem 11.14]{HHM2}. We include it here in order to clarify a few points
in the argument, and make a few simplifications.  In particular, parts of the argument are made easier due to having definable Skolem 
functions in the main sort.  Furthermore, the use of \cite[Theorem 10.15]{HHM2} is unnecessary.

As in the proof of \cite[Theorem 11.14]{HHM2},
we may assume that $e$ has the same definable closure over $C$ as some pair $(a,b)$ with 
$a\in B_n(K)/B_n(V_K)$ and $b\in B_m(K)/B_{mm}(V_K)$. By Proposition~\ref{opaque-groups} iii), these groups are opaquely layered,
and hence so is $\tp(ab/C)$ (using also \cite[Lemma 11.6]{HHM2}). By Proposition~\ref{pre-resolution-exists}, $C\langle e\rangle$
has an atomic resolution. This resolution is prime because $\R$ has definable Skolem functions.

In either of the above cases, the resolution $D$ is also minimal. For suppose there is another resolution $D'$ with $D'\subseteq D$.
By primality, $D$ embeds into $D'$ over $C\langle e\rangle$. As $D$ has finite transcendence degree over $C$ (this follows 
from the proof of Proposition~\ref{pre-resolution-exists}), $D'=D$. The fact that $D$ is unique up to isomorphism over $C\langle e\rangle$
is immediate from primality and minimality.

Finally, we show that $k_D =k(Ce)$ and $\Gamma_D=\Gamma(Ce)$. 
The inclusions $\Gamma (Ce) \subseteq \Gamma_D$ and $k(Ce)\subseteq k_D$ are obvious. To see the containment in the other
direction, consider $d\in D$. As $D$ is atomic, $\tp(d/Ce)$ is isolated, and hence also $\tp(\res(d)/Ce)$ and $\tp(v(d)/Ce)$ are
both isolated. As $k$ is real closed and $\Gamma$ is divisible abelian, the only isolated types are algebraic, so both
$\res(d)$ and $v(d)$ are in $\dcl(Ce)$. Thus $\Gamma_D \subseteq\Gamma(Ce)$ and $k_D \subseteq k(Ce)$.  All that remains is to 
note that 
$$\kInt_{\Gamma_D}^D = \acl(k_D RV(D))=\acl(k(Ce)RV(Ce))=\kInt_{\Gamma(Ce)}^{\dcl(Ce)}.
$$
\end{proof}

\subsection{Domination in the sorted structure}

We can now state our domination results in the sorted language $\G$. 

\begin{theorem}\label{sorted-domination}
Let $C$ be a substructure of $\R$ which is maximal as a valued field. Let $A$ be a definably closed subset of $\R \cup \G$, with
$A=\dcl(Ce)$ for a countable tuple of imaginaries $e\in\G$. Then
\begin{enumerate}[i)]
	\item $\tp(A/C)$ is dominated by the value group and residue field;
	\item $\tp(A/C)$ is dominated by the $k$-internal sorts over the value group.
\end{enumerate} 
\end{theorem}

\begin{proof}
i) Let $M$ be any substructure of $\R,$ containing $C$ with $k_A$, $k_M$ independent over $k_C$ and 
$\Gamma_A$, $\Gamma_M$ independent over $\Gamma_C$. We need to show that
\[
	\tp(A/Ck_M\Gamma_M) \vdash \tp(A/M).
\]
Take $\sigma \in \Aut(\mathcal{R}/Ck_M\Gamma_M )$ with $\sigma (A)=A'$.  By Theorem~\ref{resolution-exists} (and 
\cite[Corollary~11.15]{HHM2}), there is a resolution $L$ of $A$ with 
$k_L = k_A$ and $\Gamma_L= \Gamma_A$. Let $L'$ be $\sigma(L)$.  Assume $\sigma (L)=L'$ has the same residue field and value group 
as $A$. As the hypotheses of Corollary~\ref{prop-12-11-alternate} are met, we may find an automorphism $\tau$ agreeing with 
$\sigma$ on $L$ and restricting to the identity on $M$.  Since $A\subseteq \dcl(L)$, $\tau$ maps $A$ to $A'$.  Thus $A$ and 
$A'$ realize the same type over $M$.

ii) Let $C^+ = \dcl(C\Gamma_A)$ and assume that $M$ is a substructure containing $C^+$ such that 
$\kInt_{C^+}^M$ is independent from $\kInt_{C^+}^A$. 
Apply Theorem \ref{resolution-exists} to obtain a resolution $L$ of $A$ with $\kInt^L_{C^+} = \kInt^A_{C^+}$ and $\Gamma_L= \Gamma_A$.  
Apply Corollary \ref{not-alg-closed} to see that $\tp(L/C^+\kInt_{C^+}^M)\vdash \tp(L/M)$ and reason as in (i) to see that 
$\tp(A/C^+\kInt_{C^+}^M)\vdash \tp(A/M)$, as required.

\end{proof}


\end{document}